\documentclass[12pt,leqno]{amsart}
\usepackage{amsfonts, amsmath, amssymb, amsthm, geometry}

\newtheorem{theorem}{Theorem}[section]

\newtheorem{lemma}[theorem]{Lemma}
\newtheorem{proposition}[theorem]{Proposition}

\theoremstyle{definition}
\newtheorem{definition}[theorem]{Definition}

\theoremstyle{remark}
\newtheorem{remark}[theorem]{\sc Remark}

\newtheorem{example}[theorem]{\sc Example}

\begin{document}
\title[Topology of the fiber at infinity and examples of Broughton
]{Classification at infinity of polynomials of degree 3 in 3 variables}

\author[NILVA RODRIGUES RIBEIRO]{NILVA RODRIGUES RIBEIRO}

\email{nilvaufv@gmail.com}
\address[N. R. Ribeiro]{Universidade Federal de Vi\c cosa-Campus Rio
  Parana\'iba, Instituto de Ci\^encias Exatas e Tecnol\'ogicas,
  Rodovia MG-230, Km 7 Rio Parana\'iba / MG CEP: 38810-000 Caixa Postal 22.}

\begin{abstract}
We classify singularities at infinity of polynomials of degree 3 in 3 variables, $ f (x_0, x_1, x_2) = f_1 (x_0, x_1, x_2) + f_2 (x_0, x_1, x_2) + f_3 (x_0, x_1, x_2) $, $ f_i $ homogeneous polynomial of degree $ i $, $ i = 1,2,3 $. Based on this classification, we calculate the jump in the Milnor number of an isolated singularity at infinity, when we pass from the special fiber to a generic fiber. As an application of the results, we investigate the existence of global fibrations of degree 3 polynomials in $\mathbb{C}^3$
 and search for information about the topology of the fibers in each equivalence class.
\end{abstract}

\maketitle

\section{Introduction}

The study of natural fibrations of a polynomial $f: \mathbb C^n \to \mathbb C$ was introduced by Broughton \cite{Brou} long ago. At the same time,  Pham \cite{Ph} studied the conditions for a polynomial to have a good behavior at infinity,  and H\`a and L\^e \cite{HL} obtained a criterion of regularity at infinity for complex polynomials in two variables. Since then the global theory of singularities of polynomials has been developed from the point of view of this article, with contributions by \cite{mih},
\cite{mih1}, \cite{st1}, \cite{durfe}, \cite{le2}, \cite{mih3}, and others. 

In the local case, the presence of singularities is a natural obstruction for the existence of a trivial fibration associated to the germ $f$. In the global context, however,  the fibers of a polynomial can be topologically distinct, even without the presence of singularities. The values  of $ f $ for which the topology of the fiber changes are denominated {\em atypical} values. The determination of these special values depends on the behavior of $ f $ at infinity. 

In the case of polynomials in two variables, different characterizations  of atypical values are known, whereas in higher dimensions this is still an open problem. 

As in the local case, the Milnor number at infinity, and the sum  of (boundary) Milnor numbers of the generic fiber are useful invariants for studying the topology of the fiber. 

In \cite{Sier}, Siersma and Smeltink  classified the singularities at infinity of polynomials of degree 4 in two variables, obtaining conditions for the equivalence of polynomials whose homogeneous part of degree 4 are equivalent.

 In this work we use Siersma and Smeltink's method to classify singularities at infinity of polynomials 
$f:\mathbb{C}^3\rightarrow \mathbb{C}$ of degree 3 in 3 variables, $ f (x_0, x_1, x_2) = f_1 (x_0, x_1, x_2) + f_2 (x_0, x_1, x_2) + f_3 (x_0, x_1, x_2) $, $ f_i $ homogeneous polynomial of degree $ i $, $ i = 1,2,3 $. We restrict the classification to the case that all compactified fibers have only isolated singularities. Based on this classification, we study the equisingularity at infinity of the family $f=t$. We say that a polynomial $f$ is of Broughton type if $f$ has no affine singularities and the set of atypical values is non-empty.
In each equivalence class of a degree 3 polynomial in $\mathbb{C}^3$, we give conditions for the existence of examples of Broughton type. 

\section{General setting}
Consider a polynomial $f:\mathbb{C}^n\rightarrow \mathbb{C},$ $X=\{x\in \mathbb{C}^{n}; f(x)=0\}$. The zero set $X$ is an affine variety embedded in $\mathbb{C}^{n}$. Let $Sing(f)$ be the singular locus of $f,$ that is, $Sing(f)=\{x\in \mathbb{C}^{n}; grad(f)=0\}.$ We consider $\mathbb{P}^{n}$ as the standard compactification of $\mathbb{C}^{n}$ for some  fixed affine coordinates. We use the following notations: let $ \overline{X}$ be the compactification of $X$ in $\mathbb{P}^{n}$, $X^\infty = \overline{X} \cap H^\infty$ its intersection with the hyperplane at infinity $H^\infty = \mathbb{P}^{n-1}$ and $ X_t = f^{-1}(t)$.
We write  $f=f_d + f_{d-1} + \ldots + f_0$, $f_{i}$  a homogeneous polynomial of degree $i=0,\ldots, d,$  and $F=f_{d}+x_{n+1}f_{d-1}+\ldots +x_{n+1}^{d}f_{0}$ the homogenization of $f$. Then we can associate to $F$ the hypersurface  $$\mathbb X:=\{((x:x_{n+1}),t)|\in \mathbb P^n\times \mathbb C: F(x,x_{n+1})-tx_{n+1}^d=0\}.$$ The map $\tau: \mathbb{X} \to \mathbb C$ is the projection to the  $t$-coordinate and $\tau^{-1}(t)=\overline{X_t}$.

 At a point $p  \in H^\infty$ we consider the boundary pair $<\overline{X_t}, \overline{X_t}\cap H^\infty>_p$ which is a family of germs depending on $t \in \mathbb{C}$. We say that $X_t$ has a singularity at infinity if at least one of the members of this pair is singular.  Singular points of
 $\overline{X_t}$ at infinity are solutions of $grad (f_d=0)$. We can distinguish
 between two types:
 
(i) Singular points of $X_t^\infty$ where $\overline{X}_t$ is smooth. These are given by the conditions $grad(f_d)=0$ and $ f_{d-1}\neq 0$.\\
(ii) Singular points of $X_t^\infty$, where $\overline{X}_t$ is singular. These are given by the conditions $grad(f_d)=0$ and $f_{d-1} =0$.
\begin{definition}
The polynomial  $f$ is general at infinity at a point $Q$ if $ \overline{X} \pitchfork H^\infty$ at $Q$. We say $f$ is general at infinity if this condition holds for all $Q  \in \overline{X} \cap H^\infty. $
\end{definition}

\begin{definition}
The polynomial $f:\mathbb{C}^n\rightarrow \mathbb{C}$ is topologically trivial at infinity if $f$ is locally topologically trivial for all $t_0 \in \mathbb{C}.$

\end{definition}

The following definition was given in \cite{st2}.

\begin{definition}(Definition 4.1 in \cite{st2}) We define the
  following classes of polynomials.
  \begin{itemize}
  \item [\bf (i) ] We say $f$ is a {\it $ \mathcal{F}$-type} polynomial if its compactified fibers and their restrictions to the hyperplane at infinity have at most isolated singularities.
  \item [\bf (ii) ] We say $f$  is a {\it $ \mathcal{B} $-type} polynomial if its compactified fibers have at most isolated singularities.
  \end{itemize}
The $ \mathcal {F}$-class is contained in the $ \mathcal {B}$-class. Moreover, both are contained in the $ \mathcal {W }$-class, consisting of 
polynomials for which the proper extension $\tau: \mathbb{X} \rightarrow \mathbb{C} $ has only isolated singularities with respect to some  Whitney stratification of $\mathbb{X}$ such that $\mathbb{X}^\infty=\mathbb X \cap H^{\infty}$ is an union of strata, see \cite{st1}. 
\end{definition}

Based on results of \cite{st1} and \cite{st2} we can get information about the topology of the generic fiber $X_t.$ The following theorem (\cite{st1},Theorem 3.1) is known as the Bouquet Theorem.

\begin{theorem}(\cite{st1},Theorem 3.1)\label{bouq}
Let $f:\mathbb{C}^n\rightarrow \mathbb{C}$ be a polynomial with isolated $ \mathcal {W }$-singularities at infinity. Then the general fiber of $f$ is homotopy equivalent to a bouquet of spheres of real dimension $n-1$. 
\end{theorem}

Let now $(p,t) \in \mathbb{P}^{n-1} \times \mathbb{C} $ be a singular point of $\mathbb{X}_t^\infty.$ This may be a singular point of $\mathbb{X}_t$, or a point where $\mathbb{X}_t$ is non-singular but tangent to $H^\infty$ at $p$. If $(p,t)$ is an isolated singularity of $\mathbb{X}_t^\infty$, then we denote its Milnor number by $\mu_{p}^\infty$. Notice that the singularity $(p,t) \in \mathbb{X}_t^\infty $ does not depend on $t$. In contrast,  the Milnor number of the fiber $\mathbb{X}_t$ at the point $p$, that we denote by $\mu_p(\mathbb{X}_t)$, may jump at a finite number of values of $t$.  Let us denote by $\mu_{p, gen}$ the value of $\mu_{p}(t)$ for generic $t$.

For  a finite number of bifurcation values $t$ this type can change and the Milnor number can drop with a value $\lambda_p^t = \mu_p(\mathbb{X}_t)- \mu_{p}^\infty$. We denote by 
\begin{center}
$\lambda$ = the sum of all jumps in the family $f=t$.
\end{center}
In the case $f$ has only isolated singularities (in the affine space), we denote by 
\begin{center}
$\mu$ = the total Milnor number of the affine singularities.
\end{center}
Notice that this invariant can be computed by the following formula $$\mu= dim_{\mathbb{C}}\frac{\mathbb{C}[x_1, \ldots, x_n]}{Jf},$$ where the ideal $Jf \subset \mathbb{C}[x_1, \ldots, x_n]$ is the Jacobian ideal of $f$ (see for instance, \cite{st3}, pg 1).

In this paper we consider degree 3 polynomials in  $\mathbb{C}^3$, so that the fibers $\mathbb{X}_t$ are in general singular surfaces in $\mathbb{C}^3$, and $\mathbb{X}_t^\infty$ are singular curves in $\mathbb{P}^2$. 

It follows from Theorem \ref{bouq} that $b_0(\mathbb{X}_t)=1$ and $b_1(\mathbb{X}_t)=0$, where $b_i(\mathbb{X}_t), i=0,1$ are Betti numbers of the generic fiber $\mathbb{X}_t.$

For the $ \mathcal {F}$-class one can combine  two formulas from  \cite{st1} and \cite{st2}, to compute the second Betti number  $b_2$ of the generic fiber:
\begin{center}
\begin{equation}\label{eq1}
b_2=\lambda + \mu = (d-1)^3 - \displaystyle\sum_i (\mu_{p_i}gen + \mu_{p_i}^\infty)
\end{equation}
\end{center}
The right hand side can be computed via boundary data. In the left hand side $\lambda$ is the sum of all jumps in the family $f=t$. This  makes it possible to compute not only $b_2$ but also $\mu$. A similar formula exists for $ \mathcal{B}$-type, and we refer to \cite{st2}, pg 663-664.
\begin{center}
\begin{equation}\label{eq2}
b_{2}(G)=\lambda + \mu = (\chi^{3,d}-1)-\displaystyle\sum_{x \in \sum}\mu_{x,gen} - \chi^\infty 
\end{equation}

\end{center}
where $G$ is the generic fiber, and $\chi^{3,d}$ is the Euler characteristic of the smooth hypersurface $V_{gen}^{3,d}$ of degree  $d$ in $\mathbb{P}^3$ of the generic fiber $G$ of $f$ and $\chi^{\infty}= \chi(\{f_d(x)=0\}).$ In general, the following formula holds (\cite{st3}, pg 8)
$$\chi^{n,d}= \chi V_{gen}^{n,d}= n + 1 - \frac{1}{d}\{1 + (-1)^n(d-1)^{n+1}\}.$$ 

We shall denote by $Atyp (f)$ the set of atypical fibers of $f.$ It is known that $Atyp(f)=f(Sing(f))\cup B_{\infty}(f),$ where $B_{\infty}(f)$ comes from the contribution of singularities at infinity.

\section{Classification of polynomials of degree 3} 

The purpose of this section is to classify singularities at infinity of polynomials $\mathbb{C}^3 \rightarrow \mathbb{C}$  of degree 3  of the form
\begin{center}
 $f(x_0, x_1, x_2) = f_1(x_0, x_1, x_2)+
f_2(x_0, x_1, x_2) + f_3(x_0, x_1, x_2),$ 
\end{center}
$f_i$ is homogeneous polynomial of degree $i.$ We write $f_1 (x_0,x_1,x_2)=a_0x_0 + a_1x_1 + a_2x_2$, $f_2 (x_0,x_1,x_2)=a_3x_0^2 + a_4x_0x_1 + a_5x_0x_2 + a_6x_1^2 + a_7x_1x_2 + a_8x_2^2.$ Let $t \in \mathbb{C},$ the homogenization $F$ of $f-t=0$ is given by

\begin{center}
 $F(x_0, x_1, x_2,x_3) = x_3^2f_1(x_0, x_1, x_2)+
x_3f_2(x_0, x_1, x_2) + f_3(x_0, x_1, x_2) - tx_3^3.$ 
\end{center}

\begin{definition}
We say that $f$ is affine equivalent to  $g$ $ (f  \approx  g) $ if there exist linear  affine transformations $T:\mathbb{C}^3\rightarrow \mathbb{C}^3$, $L: \mathbb{C}\rightarrow \mathbb{C}$, such that $g= L\circ f \circ T^{-1}$.
\end{definition}
Notice that if $f$ and $g$ are equivalent, the linear transformation $L$ sends the fibers $f=t$ of $f$  to the fibers $g=L(t)$ of $g.$

Our aim in this section  is to classify the singularities at infinity of the fibers $f=t.$  We give  a complete classification of polynomials of type $ \mathcal {F}$  in Theorem \ref{nod} and the polynomials of type $ \mathcal {B}$ are classified in Proposition \ref{rdrs}.

We start the classification by making linear changes of coordinates to reduce the homogeneous polynomial $f_3$ to one of the following normal forms (see \cite{bri} or \cite{bru}).

\vspace{0.2cm} \label{cubica}
 \begin{enumerate}
\item [(a)]  general: $x_0^3 + x_1^3 + x_2^3 + x_0x_1x_2$
\item [(b)] nodal: $x_0^3 + x_1^3 + x_0x_1x_2$.
\item [(c)] cuspidal: $-x_0^3  + x_2x_1^2$.
\item [(d)] conic plus tangent: $(x_0^2+ x_1x_2)x_1$. 
\item [(e)] conic plus chord: $(x_0^2 + x_1x_2)x_0$. 
\item [(f)] three concurrent lines: $x_0^3+ x_1^3$ 
\item [(g)] triangle: $x_0x_1x_2$. 
\item [(h)] double line plus simple line: $x_0x_1^2$. 
\item [(i)] triple line: $x_1^3$.
 \end{enumerate}

  Let $\overline X_t \subset \mathbb P^3$ be the cubic surface defined by $F - tx_3^3=0$.\\
 Note that affine equivalences extend to the projective space sending $H^\infty$ to $H^\infty$. Translations act as the identity on $H^\infty$. The compatification $\overline X_t$ is sent to $\overline {g^{-1}(t)}$ biholomorphically. The types of local singularities do not change by affine equivalences.



Our classification is based on changes of coordinates and the recognition principles of singularities of function germs that we review in section \ref{recog}.

\subsection{Recognition of simple singularities} Set $g:(\mathbb C^3,0)\rightarrow (\mathbb C,0) $ be a holomorphic function germ at the origin. \label{recog}

We recall the normal forms of simple singularities of germs of functions \\ $g:(\mathbb{C}^3,0)\to (\mathbb{C},0),$ due to Arnol'd \cite{arnold}.
 \begin{enumerate}
\item [] $A_k: x^{k+1} + y^2 + z^2; k \geq 1$
\item [] $D_k: x^{k-1} + xy^2 + z^2; k \geq 4$
\item [] $E_6: x^3 + y^4 + z^2$
\item [] $E_7: x^3 + xy^3 + z^2$
\item [] $E_8: x^3 + y^5 + z^2$
 \end{enumerate}
The results in this section are from Bruce and Wall \cite{bru}.

The map-germ $g:(\mathbb{C}^3,0)\rightarrow (\mathbb{C},0) $ is quasihomogeneous of type $(w_1,w_2,w_3;d)$ if $f(\lambda^{w_1}x, \lambda^{w_2}y, \lambda^{w_3}z)= \lambda^d f(x,y,z)$. The normal forms of simple singularities are quasi homogeneous of the  following types:
 \begin{enumerate}
\item [] $A_k: (2,k+1,k+1; 2k +2) (k \geq 1)$
\item [] $D_k: (2, k-2, k-1;2k-2) (k \geq 4)$
\item [] $E_6: (4,3,6;12)$
\item [] $E_7: (6,4,8;18)$
\item [] $E_8: (10,6,15;30)$
 \end{enumerate}
A function $f$ is semiquasihomogeneous with respect to the weights $(w_1,w_2,w_3;d)$ if all terms of weight $< d $ in its Taylor expansion vanish and those of weight $d$ define a function with an isolated singularity.
 \begin{lemma} [\cite{bru}, Lemma 1(a)]\label{split} If $f(x,y,z)$ is semiquasihomogeneous with respect to one of the sets of weights above we can, by change of coordinates, reduce the terms of weight d to the normal forms for $A_k, D_k, E_6, E_7  $ or $E_8$ given above,and the resulting function will remain semiquasihomogeneous.
 \end{lemma}

We also quote the following Lemma \cite{bru}.

\begin{lemma} [\cite{bru}, Lemma 4] \label{bino} Let $f=x_0^2 + f_3(x_0,x_1,x_2).$ If $x_0=0$ cuts $f_3=0$ in 3 distinct lines, a double and a simple line, respectively a triple line, then $f=0$ has a $D_4, D_5$ respectively $E_6$ singularity at $0$ and no others. Also $f$ has two possible normal forms for the $D_4$ case, and a unique normal form for $D_5$ and $E_6$.  
 \end{lemma}
\subsection{The classification  theorem of polynomials of type $ \mathcal {F}$ }

In Theorem \ref{nod} we classify the fibers $f=t$  of polynomials $f$ for which the degree three homogeneous part $f_3$ has respectively 1, 2 or 3 singularities at infinity. The results of this theorem along with the results of the next section, are listed in the tables 1-6, we use the notation $\infty$ to indicate that the singularity is non-isolated.
 \begin{theorem}\label{nod}
 \begin{enumerate}
\item   [(a)] Let $f_3$ be general, $f_3 \thickapprox x_0^3 + x_1^3 + x_2^3 + \lambda x_0x_1x_2,  \lambda \neq -1$, then $f$ is general at infinity.

\item   [(b)] Let $f_3$ be nodal, $f_3 \approx x_0^3 + x_1^3 + x_0x_1x_2$.Then,  $f$ is affine equivalent to  $$a_0x_0 + a_1x_1 + a_2x_2 + a_5x_0x_2 + a_7x_1x_2 + a_8x_2^2 + x_0^3 + x_1^3 + x_0x_1x_2.$$ In this case, $\overline X_t$ is smooth at infinity or $Q=(0:0:1:0)$ is a singular point at infinity of type $A_k, 1 \leq k \leq 5$ or $Q$ is a non-isolated singularity. The conditions for each singularity type are given in Table 1.

\item [(c)] Let $f_3$ be cuspidal, $f_3 \approx -x_0^3 +x_2x_1^2$. Then $f$ is affine equivalent to $$a_0x_0 + a_1 x_1 + a_2x_2 + a_4x_0x_1 + a_5x_0x_2 + a_8x_2^2 - x_0^3 + x_2x_1^2.$$ The following conditions hold: $\overline X_t$ is smooth at infinity or $Q=(0:0:1:0)$ is a singular point at infinity of type $A_1$, $A_2$, $D_4, D_5$, $E_6$ or $Q$ is a non-isolated singularity. The conditions for each singularity type are  given in Table 2.

\item [(d)]  Let $f_3$ be conic plus tangent, $f_3 \approx x_2x_1^2 + x_0^2x_1$.Then $f$ is affine equivalent to $$ a_0 x_0 + a_1 x_1 + a_2x_2 + a_5x_0x_2 + a_7x_1x_2 + a_8x_2^2 + x_2x_1^2 + x_0^2x_1.$$ It follows that $\overline X_t$ is smooth at infinity or $Q=(0:0:1:0)$ is a singular point at infinity of type $A_1, A_3$, $D_4, D_5$, $E_6$  or $Q$ is a  non-isolated singularity. The conditions for each singularity type are given in Table 3.

\item [(e)]  Let $f_3$ be three concurrent lines, $f_3 \approx x_0^3 + x_1^3$.Then  $f$ is affine equivalent to $$a_0x_0 + a_1x_1 + a_2x_2 + a_4x_0x_1 + a_5x_0x_2  + a_7x_1x_2 + a_8x_2^2 + x_0^3 + x_1^3.$$ It follows that $\overline X_t$ is smooth at infinity or $Q=(0:0:1:0)$ is a singular point at infinity of type $A_k; 2\leq k\leq 5$, $D_4$ or $Q$ is a  non-isolated singularity. The conditions for each singularity are given in Table 4.

\item [(f)] Let $f_3$ be conic plus chord, $f_3 \approx  x_0^3 +
  x_0x_1x_2$.Then $f$ is affine equivalent to $$a_0x_0 + a_1x_1 + a_2x_2 + a_3x_0^2 + a_6x_1^2 + a_8x_2^2  + x_0^3  +
  x_0x_1x_2.$$ Then $\overline X_t$ is smooth at infinity or $Q=(0:0:1:0)$ and
  $R=(0:1:0:0)$ are singular points at infinity. It follows that $Q$ and
   $R$ is a singularity of type $A_1A_0$, $A_1A_1,$ $A_2A_0$, $A_2A_1$,  $A_3A_0$, $A_3A_1$, $A_4A_0$, $A_4A_1$, $A_5A_0,$
  $A_5A_1$ or  a  non-isolated singularity. In Table 5 we give the conditions for
  each singularity type of the points $Q$ and $R$.
		
\item [ (g)]  Let $f_3$ is triangle, $f_3\approx  x_0x_1x_2$. Then $f$ is affine
  equivalent to $$a_0x_0 + a_1x_1 + a_2x_2 + a_3x_0^2 + a_6x_1^2 + a_8x_2^2  + x_0x_1x_2.$$ It follows that $\overline X_t$ is smooth at
  infinity or $Q=(0:0:1:0)$, $R=(0:1:0:0)$ or $S=(1:0:0:0)$ are singular
  points at infinity. The singularities of $Q, R$ and $S$ are of type $A_1A_0A_0,$ $A_2A_0A_0,$ $A_3A_0A_0,$ $A_4A_0A_0,$ $A_1A_1A_0,$ $A_1A_1A_1,$ $A_2A_1A_0,$ $A_2A_1A_1,$ $A_3A_0A_0,$ $A_4A_0A_0,$ $A_3A_1A_0,$ $A_3A_1A_1,$  or else a  non-isolated singularity. In Table 6 we give the conditions for each singularity type of  the point $Q.$ 
\end{enumerate}
 \end{theorem}
\begin{remark}
In the proof below we always keep the same notation $x_0, x_1,x_2, x_3$ for the variables, even after changing  coordinates.
\end{remark}

 \begin{proof}
{ \bf (a)} Is clear from definition.

\vspace{0.2cm}

{ \bf (b)} If $f_3$ is nodal, making changes of coordinates $X_0= x_0 + h_0, X_1= x_1 + h_1$ and $X_2=x_2 + h_2$, we can
   eliminate the quadratic terms $a_3x_0^2, a_4x_0x_1$ and $a_6x_1^2$. The   homogenization $F$ of
   $f$ is given by: 

 $$F=a_0x_0x_3^2 + a_1x_1x_3^2 + a_2x_2x_3^2 + x_3( a_5x_0x_2  +
 a_7x_1x_2 + a_8x_2^2) + x_0^3 + x_1^3 + x_0x_1x_2 - tx_3^3.
 $$

 It is easy to verify that if $a_8 \neq 0$, $\overline X_t$ is smooth at infinity. If $a_8 =0$, $Q=(0:0:1:0)$ is the only singular point of $\overline X_t$ at infinity. So if $a_8 =0$ and $x_2=1$, we have:
$$
  F(x_0, x_1,1,x_3)=a_0x_0x_3^2 + a_1x_1x_3^2 + a_2x_3^2  + a_5x_0x_3
  + a_7x_1x_3 + x_0^3 + x_1^3 + x_0x_1 - tx_3^3.$$

  Let $a_2 \neq a_5a_7$. Then $Q=(0:0:1:0)$  is a singularity of type $A_1$.\\
 Let $a_2=a_5a_7$. Then $$ F(x_0,x_1,1,x_3)= (a_5x_3 + x_1)(x_0 + a_7x_3) + a_0x_0x_3^2 + a_1x_1x_3^2 + x_0^3 +x_1^3 - tx_3^3.$$  We make the change of coordinates 
 $X_0 = x_0 + a_7x_3$, $X_1=x_1 + a_5x_3$, $X_3=x_3$, and keeping the same notation $x_i, i= 0,1,2,3$ for the new coordinates, we get
$F= x_0x_1 + a_0(x_0 -a_7x_3)x_3^2 + a_1(x_1-a_5x_3)x_3^2 + (x_0
-a_7x_3)^3 + (x_1 - a_5x_3)^3 - tx_3^3 =0.$ 
If $\gamma \neq t$, where $$\gamma=-a_0a_7 -a_1a_5 -a_7^3 - a_5^3,$$ then $Q$ is a  singularity  of type $A_2$. If $\gamma = t$, giving weights $(4,4,2;8)$  it follows that if $a_0 \neq -3a_7^2$ and $a_1 \neq -3a_5^2$, then $Q$ is a singularity of type $A_3$. If $\gamma = t$, $a_0 = -3a_7^2, a_7 \neq 0$ and $a_1 \neq -3a_5^2$ ( similarly $\gamma = t$, $a_0 \neq -3a_7^2$, $a_5 \neq 0$  and $a_1 = -3a_5^2$ ),  giving weights $(5,5,2;10)$ it follows that $Q$ is a singularity of type $A_4$. If $\gamma = t$, $a_0 = -3a_7^2, a_7 = 0$ and $a_1 \neq -3a_5^2$ ( similarly $\gamma = t$, $a_0 \neq -3a_7^2$, $a_5 = 0$  and $a_1 = -3a_5^2$ ), giving weights $(6,6,2;12)$ we get that $Q$ is a singularity of type $A_5$. If $\gamma = t$, $a_0 = -3a_7^2$ and $a_1= -3a_5^2$, $Q$ is  a  non-isolated singularity. See Table 1.

{ \bf (c)} If $f_3$ is cuspidal, making changes of coordinates $X_0= x_0 + h_0, X_1= x_1 + h_1$ and $X_2=x_2 + h_2$, we can
   eliminate the quadratic terms $a_3x_0^2, a_6x_1^2$ and $a_7x_1x_2$. The   homogenization $F$ of
   $f$ is given by: 
 $$F=a_0x_0x_3^2 + a_1x_1x_3^2 + a_2x_2x_3^2 + x_3( a_4x_0x_1 +
 a_5x_0x_2 + a_8x_2^2) - x_0^3  + x_1^2x_2 - tx_3^3.$$
 
 It is easy to verify that if $a_8 \neq 0$, $\overline X_t$ is smooth at infinity. If $a_8 =0$, $Q=(0:0:1:0)$ is the only singular point of $\overline X_t$ at infinity. So if $a_8 =0$ and $x_2=1$, we have:
 $$ F(x_0, x_1,1,x_3)=a_0x_0x_3^2 + a_1x_1x_3^2 + a_2x_3^2  +a_4x_0x_1x_3 +  a_5x_0x_3 - x_0^3 + x_1^2 - tx_3^3.$$

Let $a_5 \neq 0$. Then $Q=(0:0:1:0)$  is a singularity of type $A_1$.

Let $a_5 =0$ and $a_2 \neq 0$, then $Q$ is a singularity of type $A_2$. If $a_2=a_5=0,$
then $F= x_1^2 + x_1q_2(x_0x_3) + q_3(x_0,x_3),$ where $q_2(x_0,x_3)=x_3^2 + a_4x_0x_3$ and $q_3(x_0,x_3)= a_0x_0x_3^2 - x_0^3 -tx_3^3$.  The discriminant of the cubic $q_3$ is $D(q_3)= 27t^2 - 4 a_0^3.$ If $D \neq 0$, then $q_3=0$ factors into 3 different lines and $Q$ has type $D_4.$ When $D=0$ and $a_0\neq 0$, the cubic $q_3$ has a double line and a simple line. Let
$\delta= 27a_1^6-a_0^3a_4^6.$ In this case, we have the following possibilities:\\
(i) $D=0, a_0\neq 0$ and $\delta \neq 0,$ then $Q$ has type $D_5$ for 2 different values of $t$;\\
(ii) $D=0$, $a_0 \neq 0$ and $\delta =0$, then the singularity is non isolated.\\
When $D=0$ and $a_0=0$, the cubic $q_3$ has a triple line.  In this case, if $a_1\neq 0,$ $Q$ has type $E_6,$
and if $a_1=0,$ $Q$ is  a  non-isolated singularity. See Table 2.

{ \bf (d)} If $f_3$ is conic plus tangent, making changes of coordinates $X_0= x_0 + h_0, X_1= x_1 + h_1$ and $X_2=x_2 + h_2$, we can
   eliminate the quadratic terms $a_3x_0^2, a_4x_0x_1 $ and $a_6x_1^2$. The   homogenization $F$ of
   $f$ is given by: 
 $$F=a_0x_0x_3^2 + a_1x_1x_3^2 + a_2x_2x_3^2 + x_3(  a_5x_0x_2 + a_7x_1x_2 + a_8x_2^2) + x_0^2x_1 + x_1^2x_2 - tx_3^3.
 $$
 It is easy to verify that if $a_8 \neq 0$, $\overline X_t$ is smooth at infinity. If $a_8 =0$, $Q=(0:0:1:0)$ is the only singular point of $\overline X_t$ at infinity. So if $a_8 =0$ and $x_2=1$, we have:
 $$ F(x_0, x_1,1,x_3)=a_0x_0x_3^2 + a_1x_1x_3^2 + a_2x_3^2   +  a_5x_0x_3  + a_7x_1x_3 + x_0^2x_1 + x_1^2x_2 - tx_3^3.$$

Let $a_5 \neq 0$. Then $Q=(0:0:1:0)$  is a singularity of type $A_1$.\\
Let $a_5 =0$, completing square and making the change $X_1= x_1 + \frac{a_7}{2}x_3$, we get that if $a_2- \frac{a_7^2}{4} \neq 0$, then $Q$ is a singularity of type $A_3$. If $a_2- \frac{a_7^2}{4} = 0$,  making changes of coordinates $X_1=x_1 - \frac{a_7}{2}x_3$ and giving weights $(2,3,2;6)$, then $F= x_1^2 + q_3(x_0,x_3),$ where $$q_3(x_0,x_3)=a_0x_0x_3^2 - \frac{a_7}{2}x_0^2x_3-\frac{a_1a_7}{2}x_3^3 - tx_3^3.$$ Analyzing the discriminant $D(q_3)$ of $q_3$ we have $$D(q_3)= a_7^2(a_0^2-2ta_7-a_1a_7^2)=0 \Rightarrow t=\gamma,$$ where $$\gamma= \frac{a_0^2 - a_1a_7^2}{2a_7} \ \ \ \  if \ \ \ \ a_7 \neq 0  \ \ \ \ or  \ \ \ \ a_7=0.$$ If $a_7 \neq 0$ and $t \neq \gamma$, giving weights $(2,3,2;6)$, then $Q$ is a singularity of type $D_4$. If $a_7 \neq 0$ and $t = \gamma$, giving weights $(2,4,3;8)$, then $Q$ is a singularity of type $D_5$. If $a_7 = 0$ and $a_0 \neq 0$, then $Q$ is a singularity of type $D_5$ for all values of $t$. If $a_0=0$ and $t \neq 0$, $Q$ is a singularity  of type $E_6$. Now if $t=0$, then $Q$ is  a  non-isolated singularity. See Table 3.

{ \bf (e)} If $f_3$ is three concurrent lines, making changes of coordinates $X_0= x_0 + h_0, X_1= x_1 + h_1$ and $X_2=x_2 + h_2$, we can
   eliminate the quadratic terms $a_3x_0^2$ and $a_6x_1^2$. Note the symmetry in $x_0,x_1$. The computations below and Table 4 are up to this symmetry. The   homogenization $F$ of
   $f$ is given by: 
 $$F=a_0x_0x_3^2 + a_1x_1x_3^2 + a_2x_2x_3^2 + x_3( a_4x_0x_1 + a_5x_0x_2  +
 a_7x_1x_2 + a_8x_2^2) + x_0^3 + x_1^3  - tx_3^3.
 $$

 It is easy to verify that if $a_8 \neq 0$, $\overline X_t$ is smooth at infinity. If $a_8 =0$, $Q=(0:0:1:0)$ is the only singular point of $\overline X_t$ at infinity. So if $a_8 =0$ and $x_2=1$, we have:
 $$ F(x_0, x_1,1,x_3)=a_0x_0x_3^2 + a_1x_1x_3^2 + a_2x_3^2  + a_4x_0x_1x_3 + a_5x_0x_3  + a_7x_1x_3 + x_0^3 + x_1^3 - tx_3^3.$$

 If $a_5 \neq 0$ (similarly $a_7 \neq 0$), making changes of coordinates $X_0= a_2x_3 + a_5x_0 + a_7x_1$, to get

\begin{multline*}
F= x_0x_3 + Ax_0x_3^2 + Bx_3^3 + Cx_1^2x_3 + Dx_0x_1x_3 + Ex_1x_3^2 \\+
Fx_1^3 - Gx_0^2x_1 - Hx_0^2x_3 + Ix_0x_1^2 + x_0^3 -tx_3^3,
\end{multline*}
where $$A=\frac{a_0}{a_5} + \frac{3a_2^2}{a_5^3}, B=\frac{-a_0a_2}{a_5} - \frac{a_2^3}{a_5^3}, C=- \frac{a_4a_7}{a_5} - \frac{3a_2a_7^2}{a_5^3},  D=\frac{a_4}{a_5} + \frac{6a_2a_7}{a_5^3},$$
$$ E=\frac{-a_0a_7}{a_5} + a_1 -\frac{a_2a_4}{a_5} - \frac{3a_2^2a_7}{a_5^3},  F=1 - \frac{a_7^3}{a_5^3}, G= \frac{3a_7}{a_5^3}, H=\frac{3a_2}{a_5^3}, I=\frac{3a_7^2}{a_5^3}$$
Taking weights $(3,2,3;6)$, we have that if $F \neq 0$, $Q$ is a singularity of type $A_2$. If $a_5 \neq 0$ and $F=0$, such as $a_7 \neq 0$, giving weights $(4,2,4;8)$ it follows that if $IC \neq 0,$ $Q$ is a singularity of type $A_3$. If $a_5 \neq 0$, $a_5^3=a_7^3$ and $C =0$, making changes of coordinates $x_0= X_0- Ex_1x_3$ and $x_3= X_3 - Ix_1^2$, that $E \neq 0$, $Q$ is a singularity of type $A_4$.  If $E=0$ and $B \neq t$, $Q$ is a singularity of type $A_5$. If $B=t$, then $Q$ is  a  non-isolated singularity.

If $a_5=a_7=0$ and $a_2 \neq 0$, giving weights $(2,2,3;6)$, then $Q$ is a singularity of type $D_4$. Now if $a_2=0$, then $Q$ is  a  non-isolated singularity. See Table 4.
 
   { \bf (f)} If $f_3$ is conic plus chord, making changes of coordinates $X_0= x_0 + h_0, X_1= x_1 + h_1$ and $X_2=x_2 + h_2$, we can
   eliminate the quadratic terms $a_4x_0x_1, a_5x_0x_2$ and $a_7x_1x_2$. The   homogenization $F$ of
   $f$ is given by: 
 $$F=a_0x_0x_3^2 + a_1x_1x_3^2 + a_2x_2x_3^2 + x_3( a_3x_0^2  +
 a_6x_1^2 + a_8x_2^2) + x_0^3 + x_0x_1x_2  - tx_3^3.
 $$

 It is easy to verify that if $a_6.a_8 \neq 0$, $\overline X_t$ is smooth at infinity. If $a_8 =0$, $Q=(0:0:1:0)$ is the singular point of $\overline X_t$ at infinity. If $a_6 =0$, $R=(0:1:0:0)$ is the singular point of $\overline X_t$ at infinity. Let $a_8 =0$ and $x_2=1$, we have:
 $$ F(x_0, x_1,1,x_3)=a_0x_0x_3^2 + a_1x_1x_3^2 + a_2x_3^2  + a_3x_0^2x_3  + a_6x_1^2x_3 + x_0^3 + x_0x_1x_2  - tx_3^3.$$

If $a_2,a_6 \neq 0$ and $a_8=0$, $Q$ is a singularity of type $A_1$ and $R$ is  of type $A_0$. If $a_1,a_2 \neq 0$ and $a_6=a_8=0$, both singularities $Q$ and $R$ of type $A_1$. If $a_2=a_8=0$ and  $t \neq 0$, $Q$ is a singularity of type $A_2$, with this information if $a_1=a_6 =0$, $R$ is a singularity of type $A_2$. If $a_2=a_8=t=0$, giving weights $(4,4,2;8)$ it follows that if $a_0a_1 \neq 0$, $Q$ is a singularity of type $A_3$. If  $a_1=a_2=a_8=t=0$, giving weights $(5,5,2;10)$ it follows that if $a_0a_6 \neq 0$, $Q$ is a singularity of type $A_4$.  If  $a_0=a_1=a_2=a_8=t=0$, then $Q$ is  a singularity of type non-isolated. If  $a_1=a_2=a_6=a_8=t=0$, then $Q$ is  a singularity of type non-isolated. If  $a_0=a_2=a_8=t=0$, giving weights $(5,5,2;10)$ it follows that if $a_1a_3 \neq 0$, $Q$ is a singularity of type $A_4$. If  $a_0=a_2=a_3=a_8=t=0$, giving weights $(5,5,2;10)$ it follows that if $a_1 \neq 0$, $Q$ is a singularity of type $A_5$. If  $a_0=a_1=a_2=a_8=t=0$, then $Q$ is  a  non-isolated singularity. $Q$ and $R$  points are symmetric, see Table 5.

{ \bf (g)} If $f_3$ is triangle, making changes of coordinates $X_0= x_0 + h_0, X_1= x_1 + h_1$ and $X_2=x_2 + h_2$, we can
   eliminate the quadratic terms $a_4x_0x_1, a_5x_0x_2$ and $a_7x_1x_2$. The   homogenization $F$ of
   $f$ is given by: 
 $$F=a_0x_0x_3^2 + a_1x_1x_3^2 + a_2x_2x_3^2 + x_3( a_3x_0^2  +
 a_6x_1^2 + a_8x_2^2) + x_0x_1x_2  - tx_3^3.
 $$
 It is easy to verify that if $a_3=a_6.a_8 \neq 0$, $\overline X_t$ is smooth at infinity. If $a_8 =0$, $Q=(0:0:1:0)$ is the singular point of $\overline X_t$ at infinity. If $a_6 =0$, $R=(0:1:0:0)$ is the singular point of $\overline X_t$ at infinity. If $a_3 =0$, $S=(1:0:0:0)$ is the singular point of $\overline X_t$ at infinity. Let $a_8 =0$ and $x_2=1$, we have:
 $$ F(x_0, x_1,1,x_3)=a_0x_0x_3^2 + a_1x_1x_3^2 + a_2x_3^2  + a_3x_0^2x_3  + a_6x_1^2x_3  + x_0x_1x_2  - tx_3^3.$$

 If $a_2,a_3,a_6 \neq 0$ and $a_8=0$, $Q$ is a singularity of type $A_1$, $R$ is a singularity of type $A_0$ and $S(1:0:0:0)$ is a singularity of type $A_0$. If $a_0,a_1,a_2 \neq 0$ and $a_3=a_6=a_8=0$, all singularities $Q$, $R$ and $S$ are of type $A_1$. If $a_2=a_8=0$ and  $t \neq 0$, $Q$ is a singularity of type $A_2$, with this information if $a_1=a_6 =0$, $R$ is a singularity of type $A_2$ and $a_0=a_3=0$, $S$ is a singularity of type $A_2$. If $a_2=a_8=t=0$, giving weights $(4,4,2;8)$ it follows that if $a_0a_1 \neq 0$, $Q$ is a singularity of type $A_3$. If  $a_1=a_2=a_8=t=0$, giving weights $(5,5,2;10)$ it follows that if $a_0a_6 \neq 0$, $Q$ is a singularity of type $A_4$.  If  $a_0=a_1=a_2=a_8=t=0$, then $Q$ is  a non-isolated singularity. If  $a_1=a_2=a_6=a_8=t=0$, then $Q$ is  a non-isolated singularity. If  $a_0=a_2=a_8=t=0$, giving weights $(5,5,2;10)$ it follows that if $a_1a_3 \neq 0$, $Q$ is a singularity of type $A_4$. If  $a_0=a_2=a_3=a_8=t=0$, giving weights $(5,5,2;10)$ it follows that if $a_1 \neq 0$, $Q$ is a singularity of type $A_5$. If  $a_0=a_1=a_2=a_8=t=0$, then $Q$ is  a non-isolated singularity.$Q$, $R$ and $S$ points are symmetric, see Table 6. 
 \end{proof}

 \subsection{ Classification of polynomials of type $\mathcal {B}$}

In Proposition \ref{rdrs}  we give the classification of polynomials of degree 3, classifying   the isolated singularities at infinity of $f=f_1 + f_2 + f_3$, in the cases in which $f_3$ has non isolated singularities. We denote this class of polynomials by  $ \mathcal {B} \setminus \mathcal {F}$ .

\begin{proposition}\label{rdrs}{ \bf (a)} Let $f_3=x_0x_1^2$ (double line plus simple
  line). After a change of coordinates which  leaves invariant the cubic $f_3$, we get the following possibilities for the singular points at
  infinity:
\begin{enumerate}
\item [(1)] Two points $ Q=(1:0:0:0)$ and $R= (-a_8:0:a_5:0)$, where $Q$ has type $A_1$, $R$ has type $A_k$, $2 \leq k \leq 5$, or non-isolated singularity.
\item [(2)] One point $Q=(1:0:0:0)$ with type $A_3$, $D_4$, $D_5$ or non-isolated singularity.
\item [(3)] One point $R=(0:0:1:0)$ with type $A_4$ or $D_5$.
\end{enumerate}
{ \bf (b)} Let $f_3=x_1^3$ (triple line). After a change of coordinates which  leaves invariant the cubic $f_3$, we get that the singular points at  infinity are: 

 \begin{enumerate}
\item [(1)] If  $Q =(1:0:0:0)$ with type type $A_2$, $R=(0:0:1:0)$ with $A_2$, or non-isolated singularity.
\item [(2)] If  $Q=(1:0:0:0)$, with type $A_5$ or non-isolated
  singularity.
\end{enumerate}
\end{proposition}
\begin{proof}
{\bf (a)} Let $ f_3=x_0x_1^2$, making changes of coordinates $X_0= x_0 + h_0, X_1= x_1 + h_1$ and $X_2=x_2 + h_2$, we can
   eliminate the quadratic terms $a_4x_0x_1$ and $a_6x_1^2$. In this case, the set $grad(f_3) =0$ gives a $\mathbb{P}^1$ at infinity. That is,  $Sing{f_3}
=\mathbb{P}^1=\{(x_0:0:x_2:0), (x_0,x_2) \in \mathbb{C}^2\}.$ The singularities at infinity are the points of the intersection  $Sing{f_3} \cap
\{f_2=0\}$. Hence, they are  the solutions $f_2(x_0,0,x_2)= a_3x_0^2 +
a_5x_0x_2 + a_8x_2^2=0$. We assume that $((a_3,a_5,a_8)\neq
(0,0,0))$. We distinguish two cases:

{\bf (i)} $ a_5^2 - 4a_3a_8 \neq 0$ and {\bf (ii)} $a_5^2 - 4a_3a_8 =0.$\\
{\bf (i)} When $ a_5^2 - 4a_3a_8 \neq 0$, the polynomial
$f_2(x_0,0,x_2) =0$ has two distinct roots $(\alpha_1x_0 +
\beta_1x_2), (\alpha_2x_0 + \beta_2x_2).$

If $a_8 \neq 0$ then we can
make $x_2=\alpha_1x_0 + \beta_1x_2$ to eliminate $a_3$. In this case
$a_5 \neq 0.$  Then the solutions are $Q=(1:0:0:0)$ and
$R=(-a_8:0:a_5:0)$.
So we have
$$
  F= a_0x_0x_3^2 + a_1x_1x_3^2 + a_2x_2x_3^2  + a_5x_0x_2x_3 + a_7x_1x_2x_3 +
  a_8x_2^2x_3 + x_0x_1^2 -tx_3^3.
$$

The Hessian of $F$ at the point $Q=(1:0:0:0)$, $Hess(F) (
1:0:0:0)=\frac{a_5^2}{4}$. Since $a_5 \neq 0,$ then $Q=(1:0:0:0)$ is
always $A_1$. The  Hessian of $F$ at the
point $R=(-a_8:0:a_5:0)$, $Hess(F)(-a_8:0:a_5:0)= 2a_5a_8$, $R$ is a
singularity of type $A_1$.

If $a_8=0$,  the solutions are $Q=(1:0:0:0)$ and  $R=(0:0:1:0)$,  since $Hess(F) (
1:0:0:0)=\frac{a_5^2}{4} \neq 0$, then $Q=(1:0:0:0)$ is
always $A_1$. Now let's analyze the point $R$. Making  changes of coordinates $x_0= \frac{1}{a_5}(X_0 - a_2x_3 -a_7x_1)$ we have on the chart $x_2=1:$
\begin{multline*}
  F= X_0x_3 + \frac{a_0}{a_5}(X_0 - a_2x_3 -a_7x_1)x_3^2 + a_1x_1x_3^2
  \\+ \frac{1}{a_5}(X_0 - a_2x_3 -a_7x_1)x_1^2 +
  \frac{a_3}{a_5^2}(X_0-a_2x_3-a_7x_1)^2x_3 - tx_3^3
\end{multline*}
If $a_7 \neq 0$, $R$ is a singularity of type $A_2$, if $a_7 =0$ and $a_2 \neq 0$ $R$ is a singularity of type $A_3$. If $a_2=0$ and $a_1 \neq 0$ $R$ is a singularity of type $A_4$, if $a_1=0$ and $t\neq 0$, $R$ is a singularity of type $A_5$ and finally $t=0$, then $R$ is a non-isolated singularity.

{\bf (ii)} When $a_5^2 - 4a_3a_8 =0$, the polynomial
$f_2(x_0,0,x_2)=0$ has only one root, $\alpha x_0 + \beta x_2$.

 If $a_8 \neq 0$ then we can make $x_2=\alpha x_0 +  \beta x_2$ to eliminate $a_3$ and $a_5$. In this case, the only solution is $Q=(1:0:0:0)$ and on the chart $x_0=1:$
$$F=a_0x_3^2 + a_1x_1x_3^2 + a_2x_2x_3^2 + a_8x_2^2x_3 + x_0x_1^2 - tx_3^3$$
 Giving weights $(4,2,4;8)$, if $a_0 \neq 0$,  $Q$ is a  singularity of type $A_3$, otherwise if $a_0=0$ completing square we have
$$F= (x_1 + \frac{a_1}{2}x_3^2)- \frac{a_1^2}{4}x_3^4 + a_2x_2x_3^2 + a_8x_2^2 x_3 - tx_3^3=$$
$$x_1^2  - \frac{a_1^2}{4}x_3^4 + x_3q(x_2,x_3).$$
Discriminant of $q$ is
$$D(q)=-4a_8t- a_2^2.$$
If $D(q)\neq 0$ what is $t \neq \frac{-a_2^2}{4a_8}$ and $a_1 \neq 0$, $Q$ is a singularity of type $D_4$. If $D(q)= 0$ what is $t = \frac{-a_2^2}{4a_8}$ and $a_1 \neq 0$, $Q$ is a singularity of type $D_5$ and finally $a_1=0$, then $Q$ is a non-isolated singularity.

If $a_8=0$, $a_3 \neq 0$, then the solution is $R=(1:0:0:0).$ The calculations are similar to  the first case, we get $A_4$ if $a_7 \neq 0$ and $D_5$ when $a_7=0$. If $a_3=a_5=a_8=0$ the function is no longer of $ \mathcal {B}$-type.
	
{\bf (b)} Let $ f_3=x_1^3$, making changes of coordinates $X_0= x_0 + h_0, X_1= x_1 + h_1$ and $X_2=x_2 + h_2$, we can
   eliminate the quadratic term $a_6x_1^2$. In this case, the set $grad(f_3) =0$ gives a $\mathbb{P}^1$ at infinity. That is,  $Sing{f_3}
=\mathbb{P}^1=\{(x_0:0:x_2:0), (x_0,x_2) \in \mathbb{C}^2\}.$ The singularities at infinity are the points of the intersection  $Sing{f_3} \cap
\{f_2=0\}$. Hence, they are  the solutions $f_2(x_0,0,x_2)= a_3x_0^2 +
a_5x_0x_2 + a_8x_2^2=0$. We assume that $((a_3,a_5,a_8)\neq
(0,0,0))$. We distinguish two cases:

{\bf (i)} $ a_5^2 - 4a_3a_8 \neq 0$ and {\bf (ii)} $a_5^2 - 4a_3a_8 =0.$

{\bf (i)} When $ a_5^2 - 4a_3a_8 \neq 0$, the polynomial
$f_2(x_0,0,x_2) =0$ has two distinct roots $Q$ and $R$.  By a projective transformation, leaving invariant $x_1=0$ are can arrange $Q=(1:0:0:0)$  and $R=(0:0:1:0)$. Therefore we can assume $a_3=a_8=0.$ 
So we have
\begin{multline*}
  F= a_0x_0x_3^2 + a_1x_1x_3^2 + a_2x_2x_3^2  +  a_4 x_0x_1x_3 + a_5x_0x_2x_3 + a_7x_1x_2x_3 +
   x_1^3 -tx_3^3
\end{multline*}

The Hessian of $F$ at the point $Q$, $Hess(F) (
1:0:0:0)=0$. Then when $a_4a_5 \neq 0$, $Q=(1:0:0:0)$ is
always $A_2$. The  Hessian of $F$ at the
point $R$, $Hess(F)(0:0:1:0)= 0$, if $a_4a_5 \neq 0$, $R$ is a singularity of type $A_2$.\\
If $a_4a_5=0$ than $f$ is a non-isolated singularity.

{\bf (ii)} When $a_5^2 - 4a_3a_8 =0$, the polynomial
$f_2(x_0,0,x_2)=0$ has only one root, $Q$.

 By a projective transformation, leaving invariant $x_1=0$ are can arrange $Q=(1:0:0:0)$, $a_3=0$ and therefore also $a_5=0$, $a_8 \neq 0$. On chart $x_0=1:$
$$F=a_0x_3^2 + a_1x_1x_3^2 + a_2x_2x_3^2 + a_4x_0x_1x_3 + a_7x_1x_2x_3 +  a_8x_2^2x_3 + x_1^3 - tx_3^3.$$
If $a_4 \neq 0$, making  changes of coordinates $x_1= \frac{1}{a_4}(X_1 - a_0x_3 -a_8x_2^2)$ and giving weights $(6,2,6; 12)$, 
 $Q$ is a  singularity of type $A_5$, if $a_4 =0$  $Q$ is a non-isolated singularity.
 If $a_8 = 0$,  then $f$ is not of $ \mathcal {B}$-type.
 \end{proof}

  \section{Equisingularity at infinity}
In this section we compute the invariants of the singularities in order to study the topology of the Milnor fiber. The jump $\lambda$ on the Milnor number at infinity will play an important role in the description of the topology of the regular fiber.

A careful description of regularity conditions, equisingularity and  topological triviality at infinity has given by M. Tib\u ar in \cite{mih3} ( see also \cite{st1}, \cite{st2} and \cite{st3}).
 
	As usual, the  notation $A_k\rightarrow A_{k+1}$ means that the singularity at infinity jumped  from $A_k$ to $A_{k+1}$ for some value of the atypical set. For non-isolated singularities we replace $\lambda$ by $*$.

Using the formulas (\ref{eq1}) and (\ref{eq2}) of the section $2$ it is possible to calculate the  Betti number $b_2$ and the Milnor number $\mu$ of the generic fiber.

\begin{definition}\label{equi}
Let $f$  be a polynomial of types $ \mathcal {F}$ or $ \mathcal {B}$. We say that $f = t_0$ has no Milnor-jumps at infinity at the point $Q$ if there is a neighborhood $D$ of $t_0$ in $\mathbb{C}$, such that the jump $\lambda= \mu_{t_0}^Q- \mu_{t}^Q$ is equal to zero, $\forall t \in D$, where $\mu_t$ is the Milnor number of $F$ at the point $Q$.
\end{definition}

Applying the results of Theorem \ref{nod}  and Proposition \ref{rdrs} we can calculate
$\lambda$. Knowing $\lambda$ and using the formulas (\ref{eq1}) and
(\ref{eq2}), we can calculate $b_2$ and $\mu$.

For example if $f_3$ is
nodal and the singularity  of $Q$ is of type $A_3$ for $f=t_0$ and
$A_2$ for  $f=t, \ \ t \neq t_0$, it follows that $\lambda= 1$. From  (\ref{eq1}) we get
that $b_2= \mu + \lambda = 8- (2 + 1)= 5.$  As $\lambda =1$, we get
$\mu= 4$.

 In the case where $f$ has more than one singularity, we need
to check the possible combinations of all singularities.

 For example,
if $f_3$ is conic plus chord, let's say $R$ is a singularity of type
$A_1$ and $Q$ is a singularity of type $A_2$ for $f\neq t_0$ and $A_3$ for $f= t_0$.  Then, for $A_3A_1$
singularities, we have $\lambda=1$, $b_2=3$ and $\mu=2$.

 In Theorem 4.2 and 4.3 we apply the classification given in Theorem 3.4 and Proposition 3.5 to get information about the topology of the generic  fiber $f=t$ for polynomials of type $\mathcal {F}$ and  $\mathcal {B}$. 

\begin{theorem}
Let $f$ be of  $\mathcal {F}$-type. We consider the family $f=t$. The
following Tables $1$ to $6$ give all possibilities for the
singularities of $\overline{X}_t$ at the point $Q$ respectively $R$
and $S$ at infinity.

{\bf (i)} In all cases with singularities of types $A_0$ and $A_1$ only there are no jumps.

{\bf (ii)} All jumps ($\lambda \neq 0$) are indicated  in the tables. The $t$- values are indicated in the proof.

{\bf(iii)} If there are no jumps  ($\lambda=0$) then the family $f=t$ is equisingular at infinity.
\end{theorem}
 \begin{proof}
First, notice that $(i)$ and $(iii)$  follow easily. In fact, in all
cases with singularities of type $A_0$ and $A_1$ only, there are no
jumps, that is, $\lambda=0$. In these cases the family $f=t$ is
equisingular at infinity.

To prove $(ii)$ we follow the proofs in Theorem 3.4. Especially the places where $t$ appears in the (in)equalities gives rise to the jumps. The tables contain all necessary information. Special care is needed for combinations of several critical points at infinity $P, Q$ or $S$.

The invariants $\lambda, \mu$ and $b_2$ of the generic fiber $X_t$ take into account the  combinations of singularities $Q, R$  and $S.$
 \end{proof}

 In the following tables, each line corresponds to a class of polynomial (up to affine equivalence) with the same behavior near the boundary $H^{\infty}$. We list only cases with isolated singularities, but in some cases we also list the  ``next''  non-isolated class.

 The notation $X \rightarrow Y$ means that $X$ is the generic type,
 which jumps to $Y$ nongeneric.
 
The expression $\gamma \neq t$ in each table expresses the condition that the fibers $t \neq \gamma$ are generic and $t=\gamma$ is the exceptional fiber.
The last line of each table characterizes the values of $t$ for which the jump occurs.

\begin{table}
$$ \begin{array}[b]{|l|l|c|c|c|c|c|c|c|} \hline
\multicolumn{6}{|c|} {\text{Nodal}: f_3= x_0^3 + x_1^3 + x_0x_1x_2}\\ \hline
\multicolumn{6}{|c|} {f= a_0x_0 + a_1x_1 + a_2x_2 + a_5x_0x_2 + a_7x_1x_2 + a_8x_2^2 + x_0^3 + x_1^3 + x_0x_1x_2}\\ \hline
Q(0:0:1:0) & \neq 0 &  = 0 & \lambda & \mu &  b_2 \\\hline
A_0 & a_8 &  & 0 & 7 & 7\\\hline
A_1 & a_2 - a_5a_7 & a_8 & 0 & 6 & 6 \\\hline
A_2\rightarrow A_3 & \gamma, a_0 + 3a_7^2, a_1 + 3a_5^2 & a_8, a_2-a_5a_7 & 1 & 4 & 5 \\\hline
A_2\rightarrow A_4 & \gamma, a_1 + 3a_5^2, a_7 & a_8, a_2-a_5a_7, a_0 + 3a_7^2 & 2 & 3 & 5 \\\hline
A_2\rightarrow A_5 & \gamma, a_1 + 3a_5^2 &a_8, a_2-a_5a_7, a_0 + 3a_7^2, a_7 & 3 & 2 & 5 \\\hline
A_2\rightarrow A_{\infty} & \gamma & a_8, a_2-a_5a_7, a_0 + 3a_7^2, a_1 + 3a_5^2& $*$ & - & - \\\hline
\multicolumn{6}{|c|} {\text{All jumps occur if }  t=\gamma,
     \text{where }  \gamma= -a_0a_7 - a_1a_5 - a_7^3 - a_5^3}\\ \hline
\end{array}
$$
 \caption{}
 \end{table}

 \begin{table}
   $$
   \begin{array}[b]{|l|l|c|c|c|c|c|c|c|} \hline
\multicolumn{6}{|c|} {\text{Cuspidal: }f_3= -x_0^3 +  x_1^2x_2}\\ \hline
\multicolumn{6}{|c|} {f= a_0x_0 + a_1x_1 + a_2x_2 + a_4x_0x_1 + a_5x_0x_2  + a_8x_2^2 - x_0^3  + x_1^2x_2}\\ \hline
Q(0:0:1:0) & \neq 0 &  = 0 & \lambda & \mu &  b_2 \\\hline
A_0 & a_8 &  & 0 & 6 & 6\\\hline
A_1 & a_5 & a_8 & 0 & 5 & 5 \\\hline
A_2 &  a_2  & a_8, a_5 & 0 & 4 & 4 \\\hline
D_4\rightarrow D_5 & D,\delta & a_8, a_2, a_5 & 2 & 0 & 2 \\\hline
D_4\rightarrow \infty  & D & a_8, a_2, a_5, \delta  & - & - & - \\\hline
D_4\rightarrow E_6 & a_1 & a_8, a_2, a_5,a_0 & 2 & 0 & 2 \\\hline
\multicolumn{6}{|c|} { \text{All jumps occur if } D=27t^2 - 4a_0^3, \delta= 27a_1^6- a_0^3a_4^6}\\ \hline
\end{array}
$$
\caption{ }
\end{table}

 \begin{table}
$$ \begin{array}[b]{|l|l|l|c|c|c|c|c|c|c|c|} \hline
\multicolumn{6}{|c|} {\text{Conic plus tangent: } f_3= x_0^2x_1 + x_1^2x_2}\\ \hline
\multicolumn{6}{|c|} {f= a_0x_0 + a_1x_1 + a_2x_2 + a_5x_0x_2  + a_7x_1x_2 + a_8x_2^2 + x_0^2x_1 + x_1^2x_2}\\ \hline
Q(0:0:1:0) & \neq 0 &  = 0 & \lambda & \mu &  b_2 \\\hline
A_0 & a_8 &  & 0 & 5 & 5 \\\hline
     A_1 & a_5 & a_8 & 0 & 4 & 4 \\\hline
     A_3 &  a_2-\frac{a_7^2}{4}  & a_8, a_5 & 0 & 2 & 2 \\\hline
D_4\rightarrow D_5 & \gamma, a_7 & a_8, a_2- \frac{a_7^2}{4}, a_5 & 1 & 0 & 1 \\\hline
D_5 &   a_0 & a_8, a_2, a_5, a_7 & 0 & 0 & 0 \\\hline
 E_6 \rightarrow \infty  &   &a_8, a_2, a_5,a_0, a_7 & * & - & - \\\hline
 \multicolumn{6}{|c|} {\text{The jumps occur if }t=\gamma, \text{where  }\gamma = \frac{a_0^2-a_1a_7^2}{2a_7} \text{except  if }Q \text{is non-isolated, when  }t=0}\\ \hline
\end{array}
$$
\caption{}
\end{table}

 \begin{table}
$$ \begin{array}[b]{|l|l|c|c|c|c|c|c|c|} \hline
\multicolumn{6}{|c|} {\text{Three concurrent lines: }f_3= x_0^3 + x_1^3 }\\ \hline
\multicolumn{6}{|c|} {f= a_0x_0 + a_1x_1 + a_2x_2 + a_4x_0x_1  + a_5x_0x_2 + a_7x_1x_2 + a_8x_2^2 + x_0^3 + x_1^3}\\ \hline
Q(0:0:1:0) & \neq 0 &  = 0 & \lambda & \mu &  b_2 \\\hline
A_0 & a_8 &  & 0 & 4 & 4\\\hline

A_2 & a_5, a_5^3-a_7^3 & a_8 & 0 & 2 & 2 \\\hline
A_3 & a_5,\gamma & a_8, a_5^3-a_7^3 & 0 & 1 & 1 \\\hline
A_4 & a_5, E & a_8, a_5^3-a_7^3, \gamma & 0 & 0 & 0 \\\hline
A_5 \rightarrow \infty & a_5, B  & a_8, a_5^3-a_7^3, \gamma, E & 0 & 1 & 1 \\\hline
D_4 & a_2 & a_8, a_5,a_7& 0 & 0 & 0 \\\hline
\infty & &a_8, a_2,a_5,a_7 & - & - & - \\\hline
\multicolumn{6}{|c|} { \gamma = \frac{-a_4a_7}{a-5} -
     \frac{3a_2a_7^2}{a_5^3}$, $E=- \frac{a_4a_7}{a_5} -
     \frac{3a_2a_7^2}{a_5^3}$, $B=\frac{-a_0a_2}{a_5} -
     \frac{a_2^3}{a_5^3}. \text{\footnotesize{There are no jumps in this case}}}\\ \hline
   \end{array}
$$
   \caption{ }
 \end{table}

 \begin{table}
$$ \begin{array}[b]{|l|c|c|c|c|c|c|c|c|c|} \hline
\multicolumn{7}{|c|} {\text{Conic plus chord:} f_3= x_0^3 +  x_0x_1x_2  }\\ \hline
\multicolumn{7}{|c|} { f= a_0x_0 + a_1x_1 + a_2x_2 + a_3x_0^2 + a_6x_1^2 + a_8x_2^2 + x_0^3  + x_0x_1x_2  }\\ \hline
Q(0:0:1:0) & R(0:1:0:0) & \neq 0 & =0 & \lambda & \mu & b_2 \\\hline
A_0 & A_0 & a_6,a_8  &  & 0 & 6 & 6 \\\hline
A_1 & A_0 & a_6,a_2 & a_8 & 0 & 5 & 5 \\\hline
A_1 & A_1 & a_1,a_2 & a_6,a_8 & 0 & 4 & 4 \\\hline
A_2\rightarrow A_3 & A_0 & a_6,a_0,a_1 & a_2,a_8 & 1 & 3 & 4 \\\hline
A_2\rightarrow A_3 & A_1 & a_1, a_0 & a_2, a_6,a_8 & 1 & 2 & 3 \\\hline
A_2\rightarrow A_4 & A_0 & a_0,a_6 & a_1,a_2,a_8 & 2 & 2 & 4 \\\hline
A_2\rightarrow A_4 & A_0 & a_1,a_3,a_6 &  a_0,a_2,a_8 & 2 & 2 & 4 \\\hline
A_2\rightarrow A_4 & A_1 & a_1,a_3 & a_0,a_2,a_6,a_8 & 2 & 1 & 3 \\\hline
A_2\rightarrow A_5 & A_0 & a_1, a_6 & a_0,a_2,a_3,a_8 & 3 & 1 & 4 \\\hline
A_2\rightarrow A_5 & A_1 & a_1 & a_2,a_3, a_6,a_8 & 3 & 0 & 3 \\\hline
A_2\rightarrow \infty & \infty &  & a_1,a_2,a_3,a_6,a_8 & * & - & - \\\hline
\multicolumn{7}{|c|} { \text{ All jumps occur if }  t=0. }\\ \hline
\end{array}
$$
\caption{}
\end{table}

\begin{table}[b]
  $$
  \begin{array}[b]{|l|c|c|c|c|c|c|c|c|c|} \hline
\multicolumn{8}{|c|} {\text{Triangle: } f_3=   x_0x_1x_2 }\\ \hline
\multicolumn{8}{|c|} { f= a_0x_0 + a_1x_1 + a_2x_2 + a_3x_0^2 + a_6x_1^2 + a_8x_2^2   + x_0x_1x_2 }\\ \hline
Q(0:0:1:0) & R(0:1:0:0) & S(0:0:1:0) & \neq 0 & =0 & \lambda & \mu & b_2 \\\hline
A_0 & A_0 & A_0 & a_3,a_6,a_8  &  & 0 & 5 & 5 \\\hline
A_1 & A_0 & A_0 &a_6,a_2, a_3 & a_8 & 0 & 4 & 4 \\\hline
A_1 & A_1 & A_0 & a_0,a_2, a_6 & a_3,a_8 & 0 & 3 & 3 \\\hline
A_1 & A_1 & A_1 & a_0,a_1,a_2 & a_3,a_6,a_8 & 0 & 2 & 2 \\\hline
A_2\rightarrow A_3 & A_0 & A_0 & a_3,a_6,a_0, a_1 &  a_2,a_8 & 1 & 2 & 3 \\\hline
A_2\rightarrow A_3 & A_1 & A_0 & a_0, a_1, a_6 & a_2, a_3,a_8 & 1 & 1 & 2 \\\hline
A_2\rightarrow A_3 & A_1 & A_1 & a_0, a_1 & a_2,a_3,a_6,a_8 & 1 & 0 & 1 \\\hline
A_2\rightarrow A_4 & A_0 & A_0 &  a_0, a_3, a_6 & a_1, a_2,a_8 & 2 & 1 & 3 \\\hline
A_2\rightarrow A_4 & A_0 & A_0 & a_1,a_3,a_6 & a_0,a_2,a_8 & 2 & 1 & 3 \\\hline
\multicolumn{8}{|c|} { \text{All jumps occur if }t=0. }\\ \hline
\end{array}
$$
\caption{}
\end{table}

In Theorem 4.3,  we discuss the topology of the generic fiber of $\mathcal {B}-$ type polynomials $f= f_1 + f_2 + f_3$. The results are consequence of the formula (2) for $b_2$ in Section 2, and the following formulas for the top Betti defect, $\Delta_{n-1}(f)$, given by Siersma and Tibar  in  \cite{st1}, \cite{st2} and \cite{st3}.
$$\Delta_{n-1}(f)= (d-1)^n - b_{n-1}(f)$$
$$ \Delta_{n-1}(f) = \displaystyle \sum_{p \in \sum_f^{\infty} \cap \{f_{d-1}=0\}} \mu_p(\overline{X_0}) + (-1)^n \Delta \chi^{\infty},$$
where $$\Delta \chi^{\infty}:=\chi^{n-1,d}- \chi(\{f_d=0\}),$$
and
$$\chi^{n-1,d}= n-\frac{1}{d}\{1 + (-1)^{n-1}(d-1)^n\}$$
denotes the Euler characteristic of the smooth hypersurface $V_{gen}^{n-1,d}$ of degree  $d$ in $\mathbb{P}^{n-1}$.
\begin{theorem}
Let $f$ be of type $\mathcal {B}\setminus \mathcal {F} $-type. We consider the family $f=t$. The following Tables $7$ and $8$ give all possibilities for the singularities of $\overline{X}_t$ at a point $Q$ at infinity.
\end{theorem}

\begin{table}[b]
  $$
  \begin{array}[b]{|l|c|c|c|c|c|c|c|c|c|c|} \hline
\multicolumn{7}{|c|} {\text{double line plus simple line:} f_3= x_0x_1^2 }\\ \hline
\multicolumn{7}{|c|} { f= a_0x_0 + a_1x_1 + a_2x_2 + a_3x_0^2 + a_5x_0x_2 + a_7x_1x_2 + a_8x_2^2 + x_0x_1^2 }\\ \hline
Q(1:0:0:0) & R(-a_8:0:a_5:0) & \neq 0 & =0 & \lambda & \mu & b_2 \\\hline
A_1 & A_1 & \gamma, a_8 &  & 0 & 3 & 3 \\\hline
A_1 & A_2 & \gamma, a_7 & a_8 & 0 & 2 & 2 \\\hline
A_1 & A_3 & \gamma, a_2 & a_7,a_8 & 0 & 1 & 1 \\\hline
A_1 & A_4 & \gamma, a_1 & a_2, a_7,a_8 & 0 & 0 & 0 \\\hline
A_1 & A_5 \rightarrow \infty & \gamma & a_1,a_2,a_7,a_8 & - & - & - \\\hline
A_3 &  & a_0 & \gamma & 0 & 2 & 2 \\\hline
D_4\rightarrow D_5 &  & a_1 & \gamma & 1 & 0 & 1 \\ \hline
D_4\rightarrow \infty &  &  & a_1, \gamma & - & - & - \\ \hline
 & A_4 & a_3, a_7  & a_8,a_5 & 0 & 1 & 1 \\ \hline
 & D_5 & a_3  & a_5, a_7, a_8 & 0 & 0 & 0 \\ \hline
\multicolumn{7}{|c|} {\gamma=a_5^2 -4a_3a_8. \text{All jumps occur if }t=\frac{-a_2^2}{16a_8}. }\\ \hline
\end{array}
$$
\caption{ }
\end{table}

\begin{table}[b]
  
$$ \begin{array}[b]{|l|c|c|c|c|c|c|c|c|c|c|} \hline
\multicolumn{7}{|c|} {\text{triple line:} f_3= x_1^3 }\\ \hline
\multicolumn{7}{|c|} { f= a_0x_0 + a_1x_1 + a_2x_2 + a_3x_0^2 + a_4x_0x_1 + a_5x_0x_2 + a_7x_1x_2 + a_8x_2^2 + x_1^3 }\\ \hline
Q(1:0:0:0) & R(0:0:1:0) & \neq 0 & =0 & \lambda & \mu & b_2 \\\hline
A_2 & A_2 & \gamma & a_8 & 0 & 2 & 2 \\\hline
A_5 &  & a_4,a_8 & \gamma & 0 & 1 & 1 \\\hline
\infty &   & a_0, a_8 & \gamma, a_4 & - & - & - \\\hline
\multicolumn{7}{|c|} { \gamma=a_5^2 -4a_3a_8. \text{There are no jumps in this case}}\\ \hline
   \end{array}
   $$
\caption{ }
\end{table}

\begin{proof}
{\bf (a)} We first consider $f= f_1 + f_2 + f_3$, were $f_3(x_0, x_1, x_2)= x_0x_1^2$ and the singularities at infinity are $Q=(0:0:1:0)$ and $R= (-a_8:0: a_5:0)$. If $a_5^2 \neq 4a_3 a_8$, $Q \neq R$, it follows from Proposition 3.5  that $Q$ is a singular point of type $A_1$ and $R$ is $A_k$, $1\leq k \leq 5$. To compute $b_2(f)$, note that $\Delta \chi^{\infty}=-3$. When $(Q,R)$ is $A_1A_k$ we have $\Delta_2(f)= 1 + k + 3=k + 4\Rightarrow b_2(f) = 8-(k + 4)$. If $Q=R$, a singularity is of type $A_3, D_4$ or $D_5$, in these cases we have $\Delta_2(f)= \mu_p(\overline{X}) + 3 \Rightarrow b_2(f) = 8-(\mu_p(\overline{X}) + 3)$. The results appear in Table 7.

{\bf (b)} We first consider $f= f_1 + f_2 + f_3$, where $f_3(x_0, x_1, x_2)= x_1^3$ and the singularities at infinity are $Q=(0:0:1:0)$ and $R= (-a_8:0: a_5:0)$. If $a_5^2 \neq 4a_3 a_8$, $Q \neq R$, follow Proposition 3.5 $Q$ is a singular point of type $A_2$ and $R$ is $ A_2$. To compute $b_2(f)$, note that $\Delta \chi^{\infty}=-2$. When $(Q,R)$ is $A_2A_2$ we have $\Delta_2(f)= 2 + 2 + 2=6 \Rightarrow b_2(f) = 8-6=2$. If $Q=R$, a singularity is of type $A_5$, in these case we have $\Delta_2(f)= \mu_p(\overline{X}) + 2 \Rightarrow b_2(f) = 8-(\mu_p(\overline{X}) + 2)$. The results appear in Table 8.
\end{proof}

\section{Examples of Broughton type and global fibrations}\
In this section we assume that the singularities of all polynomials are of type $ \mathcal {F}$ or $ \mathcal{B}$ .

\begin{definition}
A polynomial $f: \mathbb{C}^n \rightarrow \mathbb{C}$ is  of Broughton type if $f$ has no affine singularities, but the set of atypical values $Atyp(f)$ is non empty.
\end{definition}

According to our notations, if $f$ has no singularities in $\mathbb{C}^3 $, then $(a_0,a_1,a_2) \neq (0,0,0)$.

\begin{theorem}
Let $f= f_1 + f_2 + f_3$, where $f$ is polynomial of degree 3 of type $ \mathcal
{F}$ or $ \mathcal{B}$ on $\mathbb{C}^3$ . If $f$ is a polynomial  of Broughton type  then $\lambda \neq 0, \mu =0$ and the following
conditions hold.
\begin{itemize}
\item [\bf (i)]  $f_3$ is cuspidal and the only singularity at infinity of two special fibers  is of type $D_5$ or a single fiber of type $E_6.$ 
\item [\bf (ii)]  $f_3$ is conic plus tangent and the only singularity at infinity of the special fibers  is of type $D_5.$ 
\item [\bf (iii)]  $f_3$ is conic plus chord and the combination of the singularities at infinity of the special fiber  is of type $A_1A_5.$ 
\item [\bf (iv)]  $f_3$ is triangle and the combination of the singularities at infinity of the special fiber  is of type $A_1A_1A_3.$ 
\item [\bf (v)]  $f_3$ is double line plus simple line and  the singularities at infinity of the special fiber  is of type $D_5.$ 

\end{itemize}
\end{theorem}

\begin{proof}
The proof follows directly  from   the Tables.
\end{proof}

\begin{example}
Let  $f_3$ be three concurrent lines or $f_3$ is nodal. Then $ f (x_0, x_1, x_2) = f_1 (x_0, x_1, x_2) + f_2 (x_0, x_1, x_2) + f_3 (x_0, x_1, x_2) $ is not a polynomial of Broughton type. 
\end{example}

\begin{proof} See the tables.
\end{proof}

\begin{theorem}
Let $f= f_1 + f_2 + f_3:\mathbb{C}^3 \rightarrow \mathbb{C}$ a polynomial of degree 3 of type $ \mathcal{F}$ or $ \mathcal{B}$. Then $f$ is  a global fibration iff $\lambda=\mu=0$, which is one of the following cases:
\item [\bf (i)]  $f_3$ is conic plus tangent and the only singularity at infinity of the special fiber is of type $D_5.$ 
\item [\bf (ii)]  $f_3$ is three concurrent lines and the only singularity at infinity of the special fiber is of type $D_4$ or $A_4.$ 
\item [\bf (iii)]  $f_3$ is double line plus simple line and the only singularity at infinity of the special fiber  is of type $A_4$ or $D_5.$

\end{theorem}
\begin{proof}
$\lambda=\mu=0$ for $ \mathcal{F}$, $ \mathcal{B}$- class $\Leftrightarrow$ global fibration follows from (\cite{st1}, Corollary 5.8).
\end{proof}

\subsection*{Acknowledgment}
Part of the results of this work are from my PhD Thesis \cite{nil} supervised by Maria Aparecida Soares Ruas
and Raimundo Nonato Ara\'ujo dos Santos, to whom I thank. I am especially grateful to Professor Maria Aparecida
Soares Ruas for her patience in checking the calculations several times. I am also very grateful
to the referee for the careful reading and suggestions that improved very much the presentation of this article.




\begin{thebibliography}{99}
\bibitem{ar1}{\sc V.I. Arnold}, {\em Critical points of functions on a manifold with boundary the simple Lie groups $B_k, C_k, F_4$ and singularities of evolutes}, Uspekhi Mat. Nauk 33  (1978) , no 5(203), 91-105, 237.

\bibitem{arnold}{\sc V.I. Arnold}, { \em Singularities of fractions and
  behaviour of polynomials at infinity},Tr. Mat. Inst. Steklova 221 (1998), 48-68.


\bibitem{bt} {\sc A. Bodin,  M. Tibar}, {\em  Topological Triviality of Families of Complex Polynomials}, Adv. Math. 199 2006, 136-150.

\bibitem{bri} {\sc E. Brieskorn and H. Knorrer}, {\em Plane algebraic curves},
  Birkhauser Verlag Basel (1986).
	
	\bibitem{Brou} {\sc S.A. Broughton}, {\em On the topology of polynomial
  hypersurfaces}, Proceedings A.M.S Symp. in Pure Math.,vol. 40, I (1983),
  165-178. 
	
	\bibitem{bru} {\sc J.W. Bruce and C. T. C. Wall}, {\em On The Classification of Cubic Surfaces}, J.London Math. Soc(2) 19 (1979), 245-256.

\bibitem{durfe}{\sc A. H. Durfee}, {\em Five definitions of critical point at
  infinity}, Progress in Math. vol. 162, (1998).
	
	\bibitem{HL} {\sc H. Hamm and  D. T. L\^e}, {\em Sur la topologie de
  polin\^omes complexes}, Acta Math. Vietnamica 9 (1984), 21-32.		
	
	\bibitem{le2}{\sc D. T. L\^e}, {\em Complex analytic functions with isolated
  singularities}, J. Algebraic Geometry, 1 (1992), 83-100.


\bibitem{Ph} {\sc F. Pham}, {\em Vanishing homologies and the $n$ variable
  saddle point method}, Arcata Proc. of Symp. in Pure Math. vol. 40, II
  (1983), 319-333.
	
	 \bibitem{nil} {\sc N. R. Ribeiro}, {\em Singularidades no infinito de fun\c c\~oes polinomiais}, Tese de Doutorado, ICMC-USP (2012).

	
\bibitem{Sier} {\sc D. Siersma and J. Smeltink}, {\em Classification of
  singularities at infinity of polynomials of degree 4 in two variables},
  Georgian Mathematical Journal, Vol. 7 , Number 1, (2000), 179-198.
	
\bibitem{st2} {\sc D. Siersma, M. Tib\u ar}, {\em  Deformations of Polynomials, Boundary Singularities and Monodromy}, Mosc. Math. J. 3 (2003), 1-19.

\bibitem{st1} {\sc D. Siersma, M. Tib\u ar}, {\em  Singularities at infinity and their vanishing cycles}, Duke Math. Journal 80 (3) (1995), 771-783.
	
\bibitem{st3} {\sc D. Siersma, M. Tib\u ar}, {\em  Singularity Exchange at the Frontier of the Space, Real and Complex Singularities},( S\~ao Carlos Workshop 2004). Trends in Mathematics, pp 327-342, Birkhauser Verlag 2006.	
		
\bibitem{mih} {\sc M. Tib\u ar}, {\em On the monodromy fibration of polynomial
  functions with singularities at infinity}, C.R. Acad. Sci. Paris, 324
  (1997), 1031-1035.
			
\bibitem{mih3} {\sc M. Tib\u ar}, {\em  Polynomials and vanishing
    cycles}, Cambridge Tracts in Mathematics, 170. Cambridge University Presss,  (2007).			

\bibitem{mih1} {\sc M. Tib\u ar}, {\em Regularity at infinity of real  and
  complex polynomial maps}, Singularity Theory, the C.T.C. Wall Anniversary
  Volume, LMS Lecture Notes Series 263 (1999), 249-264. Cambridge University
  Press.




  

  

\end{thebibliography}
\end{document}